\documentclass[11pt]{amsart}
\usepackage{amssymb, amsthm,amstext,amsfonts,amsmath,graphics}
\newtheorem{theorem}{Theorem}[section]
\newtheorem{question}[theorem]{Question}

\theoremstyle{definition} \theoremstyle{remark}
\numberwithin{equation}{section}

\newcommand{\ve}{\varepsilon}

\newcommand{\fun}[3]{#1 \colon #2 \to #3}

\newcommand{\m}[1]{\mathbb{#1}}

\newcommand{\diam}{\mathrm{diam\,}}
\newcommand{\ord}{\mathrm{ord}}
\newcommand{\cl}[2][X]{\mathrm{cl}_{#1}\!\left(#2\right)}
\newcommand{\inte}[2][X]{\mathrm{int}_{#1}\!\left(#2\right)}
\newcommand{\bd}[2][X]{\mathrm{bd}_{#1}\!\left(#2\right)}

\begin{document}

\title{Dendrites and symmetric products}

\author[Acosta]{Gerardo Acosta}
\address[G. Acosta]{Instituto de Matem\'aticas, Universidad Nacional
Aut\'onoma de M\'exico, Ciudad Universitaria, M\'{e}xico D.F., 04510, M\'exico}
\email{gacosta@matem.unam.mx}

\author[Hern\'andez-Guti\'errez]{Rodrigo Hern\'andez-Guti\'errez}
\address[R. Hern\'andez-Guti\'errez]{Instituto de Matem\'aticas, Universidad Nacional
Aut\'onoma de M\'exico, Ciudad Universitaria, M\'{e}xico D.F., 04510, M\'exico}
\email{rod@matem.unam.mx}

\author[Mart\'{\i}nez-de-la-Vega]{Ver\'onica Mart\'{\i}nez-de-la-Vega}
\address[V. Mart\'{\i}nez-de-la-Vega]{Instituto de Matem\'aticas, Universidad Nacional
Aut\'onoma de M\'exico, Ciudad Universitaria, M\'{e}xico D.F., 04510, M\'exico}
\email{vmvm@matem.unam.mx}

\date{August 15, 2008}
\begin{abstract}
For a given continuum $X$ and a natural number $n,$ we consider the
hyperspace $F_n(X)$ of all nonempty subsets of $X$ with at most $n$ points,
metrized by the Hausdorff metric. In this paper we show that if
$X$ is a dendrite whose set of end points is closed, $n \in \m{N}$ and
$Y$ is a continuum such that the hyperspaces $F_n(X)$ and $F_n(Y)$ are
homeomorphic, then $Y$ is a dendrite whose set of end points is closed.
\end{abstract}
\maketitle

\section{Introduction}

A \emph{continuum} is a nondegenerate, compact, connected metric
space. Let $\m{N}$ represent the set of positive integers. 
For a given continuum $X$ and $n \in \m{N},$ we consider the
following hyperspaces of $X$:
$$
F_n(X) = \{A \subset X \colon A \mbox{ is nonempty and has at most } n
\mbox{ points}\}
$$
\noindent and
$$
C_n(X) = \{A \subset X \colon A \mbox{ is closed, nonempty and has at most }
n \mbox{ components}\}.
$$
We call $C_n(X)$ the \emph{$n$-fold hyperspace of} $X$ and $F_n(X)$ the
\emph{$n$-th symmetric product of} $X.$ Both $F_n(X)$ and $C_n(X)$ are metrized
by the Hausdorff metric $H$ (\cite[Definition 2.1]{Illanes-Nadler}).

If two continua $X$ and $Y$ are homeomorphic, we write $X \thickapprox Y.$ Note that
if $X$ and $Y$ are continua, then $X \thickapprox Y$ if and only if $F_1(X) \thickapprox
F_1(Y).$ Let $\mathcal{G}$ be a class of continua, $n \in \m{N}$ and $X \in \mathcal{G}.$
We say that $X$ has \emph{unique hyperspace $F_n(X)$ in $\mathcal{G}$} if whenever
$Y \in \mathcal{G}$ is such that $F_n(X) \thickapprox F_n(Y),$
it follows that $X \thickapprox Y.$ Similarly, $X$ has \emph{unique hyperspace $C_n(X)$
in $\mathcal{G}$} if whenever $Y \in \mathcal{G}$ is such that $C_n(X) \thickapprox
C_n(Y),$ we have $X \thickapprox Y.$ If $\mathcal{G}$ is the class of all
continua, we simply say that $X$ has unique hyperspace $F_n(X)$ or unique hyperspace
$C_n(X),$ respectively. Note that each continuum $X$ has unique hyperspace $F_1(X).$

A \emph{dendrite} is a locally connected continuum that contains no simple closed curves.
Throughout this paper we denote by $\mathcal{D}$ the class of dendrites whose set of
end points is closed. In \cite[Theorem 10]{david1} it is shown that if $X \in \mathcal{D}$
is not an arc, then $X$ has unique hyperspace $C_1(X).$ In \cite[Theorem 3]{illanes2} that
every $X \in \mathcal{D}$ has unique hyperspace $C_2(X),$ and in \cite[Theorem 5.7]{david3}
that each $X \in \mathcal{D}$ has also unique hyperspace $C_n(X),$ for $n \geq 3.$ In
\cite[Theorem 5.1]{acosta-david} it is proved that if  the set of end points of the dendrite $Y$ is not
closed, then $Y$ does not have unique hyperspace $C_1(Y)$ in the class of dendrites. This result is not
known for $n \geq 2.$ By \cite[Lemma 11]{Acosta2} an arc $Y$ has unique hyperspace $C_1(Y)$ in the
class of dendrites, but not in the class of all continua.

In the First Workshop in Hyperspaces and Continuum Theory, celebrated in the
city of Puebla, Mexico, July 2-13, 2007, the problem to determine if every element
$X \in \mathcal{D}$ has unique hyperspace $F_n(X)$ was asked by A. Illanes.
During the workshop, the three authors of this paper showed that if $X \in \mathcal{D},$
$n \in \m{N}$ and $Y$ is a continuum such that $F_n(X) \thickapprox F_n(Y),$
then $Y \in \mathcal{D}.$ This is the main result of this paper. In the same workshop,
D. Herrera-Carrasco, M. de J. L\'opez and F. Mac\'{\i}as-Romero proved that every element
$X \in \mathcal{D}$ has unique hyperspace $F_n(X)$ in $\mathcal{D}$ (\cite[Theorem 3.5]{david4}).
Combining these results it follows that every element $X \in \mathcal{D}$ has unique hyperspace $F_n(X).$ This
is a partial positive answer to the following problem, which remains open.

\begin{question}
Let $X$ be a dendrite and $n \in \m{N} - \{1\}.$ Does $X$ have unique hyperspace $F_n(X)$?
\end{question}

\section{General Notions and Facts}

All spaces considered in this paper are assumed to be metric. For a
space $X,$ a point $x \in X$ and a positive number $\ve,$ we denote
by $B_X(x,\ve)$ the open ball in $X$ centered at $x$ and having
radius $\ve.$ If $A$ is a subset of the space $X$, we use the symbols
$\cl{A}, \inte{A}$ and $\bd{A}$ to denote the closure, the interior
and the boundary of $A$ in $X,$ respectively. We denote the diameter
of $A$ by $\diam (A),$ and the cardinality of $A$ by $|A|.$ The letter $I$ stands
for the unit interval $[0,1]$ in the real line $\m{R}.$ 

A \emph{finite graph} is a continuum that can be written as the union of finitely
many arcs, each two of which intersect in a finite set. A \emph{tree} is a finite graph
that contains no simple closed curves.

If $X$ is a continuum, $U_1,U_2,\ldots,U_m \subset X$ and $n \in \m{N}$ we
define:
$$
\langle U_1, U_2, \ldots, U_m\rangle_n = \left\{A \in F_n(X) \colon A
\subset \bigcup_{i=1}^mU_i \mbox{ and } A \cap U_i \ne \emptyset
\mbox{ for each } i\right\}.
$$
It is known that the sets of the form $\langle U_1, U_2, \ldots,U_m\rangle_n,$ where $m \in \m{N}$
and $U_1,U_2,\ldots,U_m$ are open in $X,$ form a basis of the topology of $F_n(X),$ i.e.,
a basis for the topology induced by the Hausdorff metric $H$ on $F_n(X)$
(\cite[Theorems 1.2 and 3.1]{Illanes-Nadler}).

If $n \in \m{N},$ then an \emph{$n$-cell} is a space homeomorphic to the Cartesian product $I^n.$

\begin{theorem} \label{n-celda}
Let $X$ be a continuum and $n \in \m{N}.$ Given $i \in \{1,2,\ldots,n\}$ let
$J_i$ be an arc in $X$ with end points $a_i$ and $b_i.$ If the sets
$J_1,J_2,\ldots,J_n$ are pairwise disjoint, then $\langle J_1,J_2,\ldots,J_n\rangle_n$
is an $n$-cell in $F_n(X)$ whose manifold interior is the set
$\langle J_1 - \{a_1,b_1\},\ldots,J_n -\{a_n,b_n\}\rangle_n.$
\end{theorem}
\begin{proof}
Given $(x_1,x_2,\ldots,x_n) \in J_1 \times J_2 \times \cdots \times J_n,$ let $g(x_1,x_2,\ldots,x_n) =
\{x_1,x_2,\ldots,x_n\}.$ It is easy to see that $\fun{g}{J_1 \times J_2 \times \cdots \times
J_n}{\langle J_1,J_2,\ldots,J_n\rangle_n}$ is a homeomorphism whose restriction to
$\prod_{i=1}^n\left(J_i - \{a_i,b_i\}\right)$ is a homeomorphism from $\prod_{i=1}^n\left(J_i - \{a_i,b_i\}\right)$
onto $\langle J_1 - \{a_1,b_1\},\ldots,J_n -\{a_n,b_n\}\rangle_n.$
\end{proof}

From now on, in this section, the letter $X$ represents a dendrite. For
properties of dendrites we refer the reader to \cite[Chapter 10]{Nadler}.
If $p \in X$ then by the \emph{order of $p$ in $X,$} denoted by $\ord_pX,$
we mean the Menger-Urysohn order (see \cite[Definition 9.3]{Nadler} and
\cite[(1.1), (iv), p. 88]{Whyburn}). We say that $p \in X$ is an \emph{end point of}
$X$ if $\ord_pX = 1.$ The set of all such points is denoted by $E(X).$ Let
$$
\begin{array}{lll}
E_a(X)  & = & \{p \in E(X) \colon \mbox{ there is a sequence in } E(X) - \{p\} \\
        &  &  \mbox{ that converges to } p\}.
\end{array}
$$
If $p \in E(X) - E_a(X)$ we call $p$ an \emph{isolated end point of} $X.$
If $\ord_pX = 2$ we say that $p$ is an \emph{ordinary point of} $X.$ The set of all
such points is denoted by $O(X).$ By \cite[Theorem 8, p. 302]{Kura}, $O(X)$ is dense
in $X.$ If $\ord_pX \geq 3,$ we say that $p$ is a \emph{ramification point of} $X.$
The set of all such points is denoted by $R(X).$

The following result is easy to prove.

\begin{theorem}\label{c-de-a}
Let $X$ be a dendrite and $n \in \m{N}.$ Assume that $A \in F_n(X)$ and
that $\mathcal{U}$ is a neighborhood of $A$ in $F_n(X).$ Then, for each
$k \in \m{N}$ with $|A| \leq k \leq n,$ there is $C \subset O(X)$ such that
$|C| = k$ and $C \in \mathcal{U}.$
\end{theorem}

If $p,q \in X$ and $p \ne q,$ then there is only one arc in $X$ joining $p$ and $q.$
We denote such arc by $[p,q].$ We also consider the sets $(p,q) = [p,q] - \{p,q\},$
$[p,q) = [p,q] - \{q\}$ and $(p,q] = [p,q] - \{p\}.$ Let $[p,q]$ be an arc in $X$ such
that $(p,q) \subset O(X).$ We say that $[p,q]$ is:
\begin{enumerate}
\item[a)] \emph{internal} if $p,q \in R(X);$
\item[b)] \emph{external} if one end point of $[p,q]$ is an end point
          of $X,$ and the other end point of $[p,q]$ is a ramification
          point of $X.$
\end{enumerate}

Note that if $[p,q]$ is an internal arc in $X,$ then $\inte{[p,q]} =
(p,q).$ If $[p,q]$ is an external arc in $X$ and $p \in E(X),$ then
$\inte{[p,q]} = [p,q).$

Given $n \in \m{N},$ we consider the following subsets of $F_n(X)$:
$$
EA_n(X) =  \left\{A \in F_n(X) \colon A \cap E_a(X) \ne \emptyset\right\},
$$
$$
R_n(X)  = \left\{A \in F_n(X) \colon A \cap R(X) \ne \emptyset\right\}
$$
\noindent and
$$
\Lambda_n(X) = F_n(X) - \left(R_n(X) \cup EA_n(X)\right).
$$
Note that $A \in \Lambda_n(X)$ if and only if $A \in F_n(X)$ and
$A$ is contained in $O(X)\cup \left(E(X) - E_a(X)\right).$

\section{The class $\mathcal{D}$}

Recall that $\mathcal{D}$ is the class of all dendrites whose set of end points is
closed. Let $X \in \mathcal{D}.$ By \cite[Theorem 3.3]{Paty} the order of every point
of $X$ is finite. Let us assume that $s \in X$ is the limit of a sequence $(s_n)_n$ of
distinct ramification points of $X$ and that $s \ne s_1.$ By \cite[Proposition 3.4]{Paty}
$s$ is both the limit of a sequence of ramification points of $X,$ all in the arc
$[s,s_1],$ and the limit of a sequence of end points, all different than $s.$
Now assume that $e \in X$ is the limit of a sequence $(e_n)_n$ of
distinct end points of $X$ and that $e \ne e_1.$ Then $e$ is also the limit
of a sequence of ramification points of $X,$ all in the arc $[e,e_1].$
Hence $e \in E_a(X)$ if and only if $e$ is the limit of a sequence of
ramification points of $X.$

\begin{theorem} \label{por-grafica}
Let $X \in \mathcal{D}$ and $n \in \m{N}.$ If $A \in F_n(X) -
EA_n(X),$ then there exists a tree $T$ in $X$ such that:
\begin{enumerate}
\item[$(\ast)$] $A \subset \inte{T}$ and $T \cap E_a(X) = \emptyset.$
\end{enumerate}
\end{theorem}
\begin{proof}
Let $A \in F_n(X) - EA_n(X).$ We proceed by induction over $|A|.$ If $|A| = 1,$ then $A = \{x\}.$
Let $k = \ord_{x} X.$ Since $X \in \mathcal{D},$ $k$ is finite, so $X - \{x\}$
has exactly $k$ components $C_1,C_2,\ldots,C_k.$ Since $x \notin E_a(X),$ for each
$i = \{1,2,\ldots,k\},$ there is $p_i \in O(X) \cap C_i$ such that if
$T = \bigcup_{i=1}^k[p_i,x],$ then $T - \{x\} \subset O(X).$ Hence $T$ is a tree in $X$
that satisfies $(\ast)$.

Now suppose that if $B \in F_n(X) - EA_n(X)$ contains $i$ points, with $i < n,$ then there is
a tree $G$ in $X$ that satisfies $(\ast)$, replacing $A$ by $B$ and $T$ by $G,$ respectively.
Assume that $|A| = i+1$ and let $A = \{x_1,x_2,\ldots,x_{i+1}\}.$
Let $T_1$ be a tree in $X$ such that $A - \left\{x_{i+1}\right\} \subset \inte{T_1}$ and $T_1 \cap E_a(X)
= \emptyset.$ By the first part of this proof, there exists a tree $T_2$ in $X$ such that
$x_{i+1} \in \inte{T_2}$ and $T_2 \cap E_a(X) = \emptyset.$ Thus $T = T_1 \cup T_2 \cup
[x_i,x_{i+1}]$ is a tree in $X$ that satisfies $(\ast).$
\end{proof}

\begin{theorem} \label{componentes}
Let $X \in \mathcal{D}$ and $m,n \in \m{N}$ so that $m \leq n.$ Let
$\mathcal{U} = \langle U_1,U_2,\ldots,U_m\rangle_n$ be an open subset of $F_n(X)$
such that:
\begin{enumerate}
\item[1)] $U_i$ is an open connected subset of $X,$ for each $i \in \{1,2,\ldots,m\};$
\item[2)] $U_i \cap U_j = \emptyset$ if $i,j \in \{1,2,\ldots,m\}$ and $i \ne j.$
\end{enumerate}
For each $i \in \{1,2,\ldots,m\},$ let $\{J^i_\alpha \colon \alpha \in \mathcal{A}_i\}$
be the set of components of $U_i \cap [O(X) \cup \left(E(X) - E_a(X)\right)].$ Then the
components of $\mathcal{U} \cap \Lambda_n(X)$ are the nonempty sets of the form:
$$
\left\langle J^{r_1}_{\alpha_1}, J^{r_2}_{\alpha_2}, \ldots,J^{r_k}_{\alpha_k}\right\rangle_n,
$$
\noindent where $\{r_1,r_2,\ldots,r_k\} = \{1,2,\ldots,m\},$ the sets
$J^{r_1}_{\alpha_1}, J^{r_2}_{\alpha_2}, \ldots,J^{r_k}_{\alpha_k}$ are pairwise different and
$\alpha_t \in \mathcal{A}_{r_t},$ for every $t \in \{1,2,\ldots,k\}.$
\end{theorem}
\begin{proof}
It is easy to see that for each $i \in \{1,2,\ldots,m\}$ and every $\alpha \in \mathcal{A}_i,$
$J^i_\alpha$ is an open connected subset of $X.$ Let
$J^{r_1}_{\alpha_1}, J^{r_2}_{\alpha_2}, \ldots,J^{r_k}_{\alpha_k}$ be a finite collection of
pairwise different sets such that $\{r_1,r_2,\ldots,r_k\} = \{1,2,\ldots,m\}$ and
$\alpha_t \in \mathcal{A}_{r_t},$ for every $t \in \{1,2,\ldots,k\}.$ Since the sets
$J^{r_1}_{\alpha_1}, J^{r_2}_{\alpha_2}, \ldots,J^{r_k}_{\alpha_k}$ are open and connected, by
\cite[Lemma 1]{Jorge},
$$
\left\langle J^{r_1}_{\alpha_1}, J^{r_2}_{\alpha_2}, \ldots,J^{r_k}_{\alpha_k} \right\rangle_n
$$
\noindent is an open connected subset of $F_n(X).$ Let
$J^{s_1}_{\epsilon_1}, J^{s_2}_{\epsilon_2}, \ldots,J^{s_l}_{\epsilon_l}$ be a finite collection
of pairwise different sets such that: $\{s_1,s_2,\ldots,s_l\} = \{1,2,\ldots,m\},$
$\epsilon_v \in \mathcal{A}_{s_v},$ for every $v \in \{1,2,\ldots,l\},$ and
$$
\left\{J^{r_1}_{\alpha_1}, J^{r_2}_{\alpha_2}, \ldots,J^{r_k}_{\alpha_k}\right\} \ne
\left\{J^{s_1}_{\epsilon_1}, J^{s_2}_{\epsilon_2}, \ldots,J^{s_l}_{\epsilon_l}\right\}.
$$
It is not difficult to see that:
$$
\left\langle J^{r_1}_{\alpha_1}, J^{r_2}_{\alpha_2}, \ldots,J^{r_k}_{\alpha_k}\right\rangle_n \cap
\left\langle J^{s_1}_{\epsilon_1}, J^{s_2}_{\epsilon_2}, \ldots,J^{s_l}_{\epsilon_l}\right\rangle_n= \emptyset.
$$
Now assume that $\mathcal{C}$ is a component of $\mathcal{U} \cap \Lambda_n(X)$. Note that, for every
$A \in \mathcal{C}$, there is a unique finite collection
$$
J^{s_1}_{\sigma_1}, J^{s_2}_{\sigma_2}, \ldots,J^{s_w}_{\sigma_w}
$$
\noindent of pairwise different sets such that:  $\{s_1,s_2,\ldots,s_w\} = \{1,2,\ldots,m\},$
$\sigma_j \in \mathcal{A}_{s_j},$ for each $j \in \{1,2,\ldots,w\},$ and
$$
A \in \mathcal{V}_A =
\left\langle
J^{s_1}_{\sigma_1}, J^{s_2}_{\sigma_2}, \ldots,J^{s_w}_{\sigma_w}
\right\rangle_n.
$$

Hence $\mathcal{C} = \bigcup_{A\in\mathcal{C}}{\mathcal{V}_A}$, which expresses the open connected
set $\mathcal{C}$ as a union of nonempty pairwise disjoint open connected sets. Thus $\mathcal{C}$ is
of the form $\left\langle J^{r_1}_{\alpha_1}, J^{r_2}_{\alpha_2}, \ldots,J^{r_k}_{\alpha_k}\right\rangle_n$
where $\{r_1,r_2,\ldots,r_k\} = \{1,2,\ldots,m\},$ the sets
$J^{r_1}_{\alpha_1}, J^{r_2}_{\alpha_2}, \ldots,J^{r_k}_{\alpha_k}$ are pairwise different and
$\alpha_t \in \mathcal{A}_{r_t},$ for every $t \in \{1,2,\ldots,k\}.$
\end{proof}

Assume that $X \in \mathcal{D}.$ It is not difficult to see that $\Lambda_n(X)$ is an open subset
of $F_n(X).$  As a particular case of Theorem \ref{componentes} we obtain the following result,
which is the equivalent version of \cite[Lemma 4.1]{willy-alex} for elements of $\mathcal{D}.$

\begin{theorem} \label{componente}
Let $X \in \mathcal{D}$ such that $X$ is not an arc and $n \in \m{N}.$ Then the components of
$\Lambda_n(X)$ are exactly the sets of the form:
$$
\left\langle \inte{I_1}, \inte{I_2}, \ldots, \inte{I_m}\right\rangle_n,
$$
\noindent where $m \leq n,$ $I_j$ is either an internal or an external arc in $X$
for every $j \in \{1,2,\ldots,m\},$ and the sets $\inte{I_1},$ $\inte{I_2},\ldots,\inte{I_m}$
are pairwise disjoint.
\end{theorem}

The following result is the equivalent version of \cite[Lemma 4.3]{willy-alex}, for
elements of $\mathcal{D}.$

\begin{theorem}\label{encaje}
Let $X \in \mathcal{D}$ and $n \geq 4.$ If $A \in F_{n-1}(X),$ then
no neighborhood of $A$ in $F_n(X)$ can be embedded in $\m{R}^n.$
\end{theorem}
\begin{proof}
We show first that:
\begin{enumerate}
\item[$(\ast)$] if $C \in F_{n-1}(X) - EA_n(X),$ then no neighborhood of $C$ in $F_n(X)$ can be
          embedded in $\m{R}^n.$
\end{enumerate}
To show $(\ast)$ let $C \in F_{n-1}(X) - EA_n(X)$ and assume that there is a neighborhood $\mathcal{V}$ of
$C$ in $F_n(X)$ that can be embedded in $\m{R}^n.$ By Theorem \ref{por-grafica}, there is a tree
$T$ in $X$ such that $C \subset \inte{T}$ and $T \cap E_a(X) = \emptyset.$ Then $\mathcal{V} \cap F_n(T)$
is a neighborhood of $C$ in $F_n(T)$ that can be embedded in $\m{R}^n.$ Since this contradicts
\cite[Lemma 4.3]{willy-alex}, claim $(\ast)$ holds.

To show the theorem let $A \in F_{n-1}(X).$ Assume that there is a neighborhood $\mathcal{U}$ of $A$ in
$F_n(X)$ that can be embedded in $\m{R}^n.$ By Theorem \ref{c-de-a}, there is $C \subset O(X)$ such that
$|C| = |A|$ and $C \in \inte[F_n(X)]{\mathcal{U}}.$ Since $A \in F_{n-1}(X)$ it follows that $C \in F_{n-1}(X)
- EA_n(X).$ Then, by $(\ast)$, no neighborhood of $C$ in $F_n(X)$ can be embedded in $\m{R}^n.$ However, since
$C \in \inte[F_n(X)]{\mathcal{U}},$ the set $\mathcal{U}$ is a neighborhood of $C$ in $F_n(X)$ that can be embedded in $\m{R}^n.$ This contradiction
completes the proof of the theorem.
\end{proof}

\section{The set $\mathcal{E}_n(X)$}

Given a continuum $X$ and a natural number $n,$ we consider the
following set:
$$
\mathcal{E}_n(X) = \{A \in F_n(X) \colon A \mbox{ has a neighborhood
in } F_n(X) \mbox{ which is an } n\mbox{-cell}\}.
$$
In this section we prove some properties of $\mathcal{E}_n(X).$

\begin{theorem} \label{e-0}
Let $X$ and $Y$ be continua and $n \in \m{N}.$ If $\fun{h}{F_n(X)}{F_n(Y)}$
is a homeomorphism, then $h\left(\mathcal{E}_n(X)\right) = \mathcal{E}_n(Y).$
\end{theorem}

A \emph{simple triod} is a continuum $G$ that can be written as the union of
three arcs $I_1,I_2$ and $I_3$ such that: $I_1 \cap I_2 \cap I_3 = \{p\},$ $p$
is an end point of each arc $I_i$ and $(I_i - \{p\}) \cap (I_j - \{p\}) = \emptyset,$
if $i \ne j.$ The point $p$ is called the \emph{core of $G$}.

Given a continuum $X$ let:
$$
T(X) = \{p \in X \colon p \mbox{ is the core of a simple triod
in } X\}.
$$
Let $X$ be a locally connected continuum and $A \in \mathcal{E}_n(X).$
In \cite[Lemma 3.1]{willy-alex} it is shown that $A \cap T(X) = \emptyset.$
A straightforward modification can be applied to obtain the following result.

\begin{theorem}\label{core}
Let $X$ be a locally connected continuum and $n \in \m{N}.$ If $A
\in \mathcal{E}_n(X),$ then $A \cap \cl{T(X)} = \emptyset.$
\end{theorem}

\begin{theorem} \label{e-3}
Let $X \in \mathcal{D}$ and $n \in \m{N}.$ Then $\Lambda_n(X) -
F_{n-1}(X) \subset \mathcal{E}_n(X).$
\end{theorem}
\begin{proof}
Take $A \in \Lambda_n(X) - F_{n-1}(X).$ Then $|A| = n,$ so we can write
$$
A = \{x_1,x_2,\ldots,x_n\}.
$$
Since $A \in \Lambda_n(X),$ we have
$A \subset O(X) \cup (E(X) - E_a(X)).$ Then there exist $n$
pairwise disjoint arcs $J_1,J_2,\ldots,J_n$ in $X$ such that
$x_i \in \inte{J_i},$ for each $i \in \{1,2,\ldots,n\},$ and
$$
J_1 \cup J_2 \cup \cdots \cup J_n \subset O(X) \cup (E(X) - E_a(X)).
$$
Note that $\langle J_1, J_2, \ldots,J_n\rangle_n$ is a neighborhood of
$A$ in $F_n(X)$ which is an $n$-cell, by Theorem \ref{n-celda}. Then
$A \in \mathcal{E}_n(X).$
\end{proof}

\begin{theorem} \label{en-denso}
Let $X \in \mathcal{D}$ and $n \in \m{N}.$ Then $\mathcal{E}_n(X)$
is dense in $F_n(X).$
\end{theorem}
\begin{proof}
Let $\mathcal{U}$ be a nonempty open subset of $F_n(X).$ By Theorem \ref{c-de-a}
there is $D \subset O(X)$ such that $|D| = n$ and $D \in \mathcal{U}.$ Note that
$D \in \Lambda_n(X) - F_{n-1}(X)$ so, by Theorem \ref{e-3}, $D \in \mathcal{E}_n(X).$ This
shows that $\mathcal{E}_n(X)$ is dense in $F_n(X).$
\end{proof}

\begin{theorem}\label{e-4}
Let $X \in \mathcal{D}$ and $n \in \m{N}.$ Then
\begin{enumerate}
\item[a)] $\mathcal{E}_n(X) \subset \Lambda_n(X);$
\item[b)] if $n \in \{2,3\},$ then $\mathcal{E}_n(X) = \Lambda_n(X);$
\item[c)] if $n \geq 4,$ then $\mathcal{E}_n(X) = \Lambda_n(X) - F_{n -1}(X).$
\end{enumerate}
\end{theorem}
\begin{proof}
To show a) let $A \in \mathcal{E}_n(X).$ By Theorem \ref{core},
$A \cap \cl{T(X)} = \emptyset.$ Since $X \in \mathcal{D},$ this
implies that $A \cap (R(X) \cup E_a(X)) = \emptyset.$ Thus $A \in
\Lambda_n(X),$ so a) holds. Assertion b) follows from a) and
the proof of \cite[Lemma 5.1]{willy-alex}. To show c) assume that $n \geq 4.$
Take $A \in \mathcal{E}_n(X).$ By a),
$A \in \Lambda_n(X).$ Let $\mathcal{U}$ be a neighborhood of $A$ in $F_n(X)$ which
is an $n$-cell. Then $\mathcal{U}$ can be embedded in $\m{R}^n$ so, by
Theorem \ref{encaje}, $A \notin F_{n-1}(X).$ This shows that $\mathcal{E}_n(X) \subset
\Lambda_n(X) - F_{n -1}(X).$ The other inclusion holds by Theorem \ref{e-3}.
\end{proof}

\begin{theorem} \label{importante1}
Let $X \in\mathcal{D}$ and $A \in F_n(X).$ If $A \cap E_a(X) = \emptyset,$
then there exists a basis $\mathfrak{B}$ of open neighborhoods of $A$ in $F_n(X)$ such
that for each $\mathcal{V} \in \mathfrak{B}$, the set $\mathcal{V} \cap \mathcal{E}_n(X)$
is nonempty and has a finite number of components.
\end{theorem}
\begin{proof}
Since $A \cap E_a(X) = \emptyset,$ we have $A \in F_n(X) - EA_n(X).$ Thus,
by Theorem \ref{por-grafica}, there is a tree $T$ in $X$ such that $A \subset \inte{T}$ and
$T \cap E_a(X) = \emptyset.$ Let $A = \{x_1,x_2,\ldots,x_m\}$ and consider that $A$ has
exactly $m$ points. Let $\ve > 0.$ Choose a finite collection $U_1,U_2,\ldots,U_m$ of pairwise
disjoint open connected subsets of $X$ with the following properties:

\begin{enumerate}
\item[1)] $x_i \in U_i \subset \inte{T} \cap B_X(x_i,\ve),$ for each $i \in \{1,2,\ldots,m\};$
\item[2)] $U_i - \{x_i\} \subset O(X),$ for each $i \in \{1,2,\ldots,m\}.$
\end{enumerate}

Let $\mathcal{V}_\ve = \langle U_1,U_2,\ldots,U_m\rangle_n.$ By 1) we have
$\mathcal{V}_\ve \subset B_{F_n(X)}(A,\ve)$ and, by Theorem \ref{en-denso},
$\mathcal{V}_\ve \cap \mathcal{E}_n(X) \ne \emptyset.$ Given $i \in \{1,2,\ldots,m\},$
since $X \in\mathcal{D},$ the order of $x_i$ in $X$ is finite. From this and 2),
the set
$$
U_i \cap [O(X) \cup (E(X) - E_a(X))]
$$
\noindent has a finite number of components. Let
$\left\{J^i_1,J^i_2,\ldots,J^i_{l_i}\right\}$ be the set of components of $U_i \cap [O(X) \cup (E(X) - E_a(X))].$
By Theorem \ref{componentes} the components of $\mathcal{V}_\ve \cap \Lambda_n(X)$ are the nonempty
sets of the form:
\begin{equation}\label{forma}
\left\langle J^{r_1}_{s_1}, J^{r_2}_{s_2}, \ldots,J^{r_k}_{s_k}\right\rangle_n
\end{equation}
\noindent where $\{r_1,r_2,\ldots,r_k\} = \{1,2,\ldots,m\},$ the sets
$J^{r_1}_{s_1}, J^{r_2}_{s_2}, \ldots,J^{r_k}_{s_k}$ are pairwise different and
$s_t \in \left\{1,2,\ldots,l_{r_t}\right\},$ for every $t \in \{1,2,\ldots,k\}.$ Since we have a finite
number of elements of the form $J^{r_i}_{s_t},$ the number of nonempty sets of the form (\ref{forma}) is finite.

If $n \in \{2,3\}$ then, by part b) of Theorem \ref{e-4}, the nonempty sets of the form (\ref{forma}) are the
components of $\mathcal{V}_\ve \cap \mathcal{E}_n(X).$ Assume then that $n \geq 4.$ Then, by part c) of Theorem \ref{e-4}, $\mathcal{E}_n(X) = \Lambda_n(X) - F_{n -1}(X).$ Given a component $\mathcal{C} = \left\langle J^{r_1}_{s_1}, J^{r_2}_{s_2}, \ldots,J^{r_k}_{s_k}\right\rangle_n$
of $\mathcal{V}_\ve \cap \Lambda_n(X)$ and $(q_1,q_2,\ldots,q_k) \in \m{N}^k$ such that $q_1 + q_2 + \cdots + q_k =
n$ let:
$$
\mathcal{C}(q_1,q_2,\ldots,q_k) =
\left\{C \in \mathcal{C} \colon \left|C \cap J^{r_t}_{s_t}\right| = q_t \mbox{ for each } t \in \{1,2,\ldots,k\}\right\}.
$$
Note that $\mathcal{C}(q_1,q_2,\ldots,q_k) \subset \mathcal{V}_\ve \cap \mathcal{E}_n(X).$
It is not difficult to see that $\mathcal{C}(q_1,q_2,\ldots,q_k)$ is homeomorphic to
$$
\left(F_{q_1}(J^{r_1}_{s_1}) - F_{q_1 -1}(J^{r_1}_{s_1})\right) \times
\cdots \times
\left(F_{q_k}\left(J^{r_k}_{s_k}\right) - F_{q_k -1}\left(J^{r_k}_{s_k}\right)\right),
$$
\noindent where we agree that $F_0(R) = \emptyset$ for each continuum $R.$ Since the sets
$$
F_{q_1}\left(J^{r_1}_{s_1}\right) - F_{q_1 -1}\left(J^{r_1}_{s_1}\right), \ldots,
F_{q_k}\left(J^{r_k}_{s_k}\right) - F_{q_k -1}\left(J^{r_k}_{s_k}\right)
$$
\noindent are connected, $\mathcal{C}(q_1,q_2,\ldots,q_k)$ is a connected subset of
$\mathcal{V}_\ve \cap \mathcal{E}_n(X).$
Moreover
$$
\mathcal{C} \cap \mathcal{E}_n(X) = \bigcup \left\{\mathcal{C}(q_1,\ldots,q_k) \colon
(q_1,\ldots,q_k) \in \m{N}^k \mbox{ and } q_1 + \cdots + q_k = n \right\}.
$$
This implies that $\mathcal{C} \cap \mathcal{E}_n(X)$ has a finite number of components.
Since each component of $\mathcal{C} \cap \mathcal{E}_n(X)$ is a component of
$\mathcal{V}_\ve \cap \mathcal{E}_n(X)$ and $\mathcal{V}_\ve \cap \Lambda_n(X)$ has
a finite number of components, the set $\mathcal{V}_\ve \cap \mathcal{E}_n(X)$ has
a finite number of components as well.

To finish the proof note that
$\mathfrak{B} = \{\mathcal{V}_\ve \colon \ve > 0\}$ is a basis of open neighborhoods
of $A$ in $F_n(X).$
\end{proof}

In \cite{Bing} and \cite{Moise} it is proved that locally connected continua admit a convex metric $d.$
This means that every two points $x,y \in X$ can be joined by an arc $J$ in $X,$ in such a way that
$J$ is isometric to the closed interval $[0,d(x,y)].$

\begin{theorem} \label{vecindades contraibles}
Let $X \in \mathcal{D}$ and $A \in F_n(X).$ Assume that $A \cap E(X) \ne \emptyset.$ Then
there exists a basis $\mathfrak{B}$ of open neighborhoods of $A$ in $F_n(X)$ such that, for each
$\mathcal{V} \in \mathfrak{B},$ the set $\mathcal{V} - \{A\}$ is contractible. Moreover if
$A \cap E_a(X) \ne \emptyset$ we can choose $\mathfrak{B}$ with the additional property that,
for each $\mathcal{V} \in \mathfrak{B},$ the set $\mathcal{V} \cap \mathcal{E}_n(X)$ has infinitely many
components.
\end{theorem}
\begin{proof}
Let $d$ be a convex metric on $X.$ Assume that $|A| = m.$ Let $A = \{x_1,x_2,\ldots,x_m\}$ and assume that
$x_1 \in E(X).$ Let $\ve > 0.$ Choose a finite collection $U_1,U_2,\ldots,U_m$ of pairwise disjoint
open connected subsets of $X$ such that $x_i \in U_i \subset B_X(x_i,\ve),$ for each
$i \in \{1,2,\ldots,m\}.$ Let $\mathcal{V}_\ve = \langle U_1,U_2,\ldots,U_m\rangle_n.$ Clearly
$A \in \mathcal{V}_\ve \subset B_{F_n(X)}(A,\ve).$ Assume that $\diam(U_i) < 1,$ for each
$i \in \{1,2,\ldots,m\}.$ Fix $B = \{b_1,b_2,\ldots,b_m\}$
so that $b_1 \in U_1 - \{x_1\}$ and $b_i \in U_i$ for each $i \in \{2,3,\ldots,m\}.$ Note
that $B \in \mathcal{V}_\ve - \{A\}.$ Given $i \in \{1,2,\ldots,m\}$ and $(x,t) \in
U_i \times I,$ by \cite[Theorem 8.26]{Nadler}, $[x,b_i] \subset U_i.$ We also have that
$[x,b_i]$ is isometric to the closed interval $\left[0,d(x,b_i)\right].$ Hence if
$d(x,b_i) \geq t$ there is a unique point $y_x \in [x,b_i]$ such that $d(x,y_x) = t.$
We can then define a function $\fun{g_i}{U_i \times I}{U_i}$ by:
$$
g_i(x,t) =
\left\{
  \begin{array}{ll}
    b_i, & \hbox{if } d(x,b_i) \leq t;\\
    y_x, & \hbox{if } d(x,b_i) \geq t.
  \end{array}
\right.
$$
It is not difficult to prove that $g_i$ is a continuous function. Note that
$g_i(x,0) = x$ and $g_i(x,1) = b_i,$ for all $x \in U_i.$ If $x \in U_1 - \{x_1\}$
then $[x,b_1] \subset U_1 - \{x_1\}$ so, by the definition of $g_1,$ we have
$g_1(x,t) \in U_1 - \{x_1\}$ for every $t \in I.$

Define $\fun{G}{\left(\mathcal{V}_\ve - \{A\}\right) \times I}{\mathcal{V}_\ve - \{A\}}$
so that if $(D,t) \in \left(\mathcal{V}_\ve - \{A\}\right) \times I,$ then:
$$
G(D,t) = \bigcup_{i=1}^m g_i\left((D \cap U_i) \times \{t\}\right).
$$
It is not difficult to see that $G$ is well defined and continuous. Since $G(D,0) =
\bigcup_{i=1}^m (D \cap U_i) = D$ and $G(D,1) = B,$ for each $D \in \mathcal{V}_\ve - \{A\},$
the set $\mathcal{V}_\ve - \{A\}$ is contractible.

Let us assume now that $x_1 \in E_a(X).$ Since $X \in \mathcal{D},$ each element of $E_a(X)$
is the limit of a sequence of distinct ramification points of $X,$ all in the same arc. We also
have, since $X \in \mathcal{D},$ that $R(X)$ is discrete (\cite[Corollary 3.6]{Paty}). Then we
can find a sequence $(r_k)_k$ in $R(X) \cap U_1$ such that:

\begin{enumerate}
\item[1)] $(r_k)_k$ converges to $x_1;$
\item[2)] $\left(r_{k+1},r_k\right)$ is an internal arc in $X,$ for every $k \in \m{N};$
\item[3)] $r_{k+1} \in \left(r_{k+2},r_k\right) \subset U_1,$ for each $k \in \m{N}.$
\end{enumerate}
Given $i \in \{2,3,\ldots,m\}$ fix an arc $I_i$ in $\cl{U_i}$ which is either external or
internal in $\cl{U_i}.$ Let $J_i = \inte[U_i]{I_i \cap U_i}.$ By Theorems \ref{componentes} and
\ref{e-4}, for every $k \in \m{N},$ the set:
$$
\mathcal{W}_k=
\langle J_2, J_3, \ldots, J_m, (r_{k+1},r_k),(r_{k+2},r_{k+1}),\ldots,(r_{k+n-m+1},r_{k+n-m})\rangle_n
$$
\noindent is a component of $\mathcal{V}_\ve \cap \mathcal{E}_n(X).$ Since $\mathcal{W}_k \cap
\mathcal{W}_l = \emptyset,$ if $k \ne l,$ the set $\mathcal{V}_\ve \cap \mathcal{E}_n(X)$ has
infinitely many components.

To finish the proof, note that $\mathfrak{B} = \{\mathcal{V}_\ve \colon \ve > 0\}$
is a basis of open neighborhoods of $A$ in $F_n(X)$ as required.
\end{proof}

\begin{theorem}\label{peludas}
Let $X$ be a locally connected continuum and $Z$ be a nondegenerate subcontinuum of $X$ such that
$\cl{T(X) \cap Z} = Z.$ Assume that there is a point $p\in Z$ such that $p \in \inte{Z}.$
Then there exists a basis $\mathfrak{B}$ of open neighborhoods of $\{p\}$ in $F_n(X)$ such that, for each
$\mathcal{V} \in \mathfrak{B},$ the set $\mathcal{V} \cap \mathcal{E}_{n}(X)$ is empty.
\end{theorem}
\begin{proof}
Take $\ve > 0$ such that $B_X(p,\ve) \subset \inte{Z}.$ Let
$$
\mathfrak{B} = \left\{B_{F_n(X)}(\{p\},\delta) \colon \delta < \ve\right\},
$$
\noindent $\mathcal{V} \in \mathfrak{B}$ and $A \in \mathcal{V} \cap F_n(Z).$
Since $\cl{T(X) \cap Z} = Z,$ we have $A \cap \cl{T(X)} \ne \emptyset.$ Thus,
by Theorem \ref{core}, $A \notin \mathcal{E}_{n}(X).$ This implies that $\mathcal{V} \cap
\mathcal{E}_{n}(X) = \emptyset.$
\end{proof}

Let $X$ be a continuum and $A$ be an arc in $X$ with end
points $p$ and $q.$ We say that $A$ is a \emph{free arc of} $X$ if
$A - \{p,q\}$ is an open subset of $X.$

\begin{theorem}\label{arco-libre}
Let $X$ be a locally connected continuum and $n \in \m{N}$ such that
$\mathcal{E}_n(X)$ is dense in $F_n(X).$ Then, for each nonempty open subset
$U$ of $X,$ there is a free arc of $X$ contained in $U.$
\end{theorem}
\begin{proof}
Assume, to the contrary, that $U$ contains no free arcs. Let $V$ be a nonempty open connected subset of $X$ such that 
$\cl{V} \subset U.$ Define $Z = \cl{V}.$ We prove that $Z = \cl{T(X) \cap Z}.$ Let $y \in Z$ and $W$ be an 
open subset of $X$ such that $y \in W.$ Let $p \in W \cap V$ and $A$ be an arc such that 
$p \in A \subset V \cap W.$ Since $U$ has no free arcs and open subsets of $X$ are locally arcwise connected 
(see \cite[Definition 8.24 and Theorem 8.25]{Nadler}) it can be shown that there is $a \in A \cap T(X) \cap W.$ 
Thus $W \cap T(X) \cap Z \ne \emptyset.$ This shows that $Z \subset \cl{T(X) \cap Z}$ and, since the other 
inclusion also holds, we have $\cl{T(X) \cap Z} = Z.$ Since the interior of $Z$ is nonempty, by Theorem \ref{peludas}, 
there is an open set $\mathcal{V}$ in $F_n(X)$ such that $\mathcal{V} \cap \mathcal{E}_{n}(X) = \emptyset.$ This 
contradicts the fact that $\mathcal{E}_n(X)$ is dense in $F_n(X).$ Therefore $U$ contains a free arc.
\end{proof}

\section{The Main Theorem}

We start this section by showing the following result, which is a positive answer to
\cite[Question 2]{illanes}. In its proof we will use the fact that a continuum $Z$ is
locally connected if and only if $F_n(Z)$ is locally connected (\cite[Theorem 6.3]{Char-Ill}),
and also that if $Z$ is a one-dimensional continuum, then $\dim(F_n(Z)) = n.$ This
follows from \cite[Theorem 3]{borsuk} and \cite[Proof of Lemma 3.1]{Curtis}.

\begin{theorem} \label{primer-paso}
Let $X$ be a dendrite and $n \in \m{N}.$ If $Y$ is a continuum such that
$F_n(X) \thickapprox F_n(Y),$ then $Y$ is a dendrite.
\end{theorem}
\begin{proof}
Since $X$ is locally connected, $Y$ is also locally connected. By \cite[Theorem 1.1(19)]{Chara2},
$\dim(X) = 1.$ Thus $\dim(F_n(Y)) = \dim(F_n(X)) = n.$ Assume that $\dim(Y) > 1.$ Then there
exist $q \in Y$ and a compact neighborhood $B$ of $q$ such that $\dim(B) \geq 2.$ Such $B$
can be chosen so that there is a finite collection $A_1,A_2,\ldots,A_{n-1}$ of pairwise disjoint
arcs in $Y$ such that $B \cap A_i = \emptyset,$ for all $i \in \{1,2,\ldots,n-1\}.$ Then $\mathcal{B}=
\langle B,A_1,A_2,\ldots,A_{n-1}\rangle_n$ is a subset  of $F_n(Y)$ which is homeomorphic to
$B \times A_1 \times A_2 \times \cdots \times A_{n-1}.$ Since $B$ is compact and $\dim(A_i) = 1$
for every $i \in \{1,2,\ldots,n-1\},$ by \cite[Remark, p. 34]{Hurewicz}, $\dim(\mathcal{B}) =
\dim(B \times A_1 \times A_2 \times \cdots \times A_{n-1}) = \dim(B) + n -1 \geq n +1.$ Hence
$\dim(F_n(Y)) \geq n + 1.$ Since this is a contradiction, $\dim(Y) = 1.$

Assume that $Y$ contains a simple closed curve $S^1.$ Since $\dim(Y) = 1$, by
\cite[18.8, p. 104]{Nadler-Dim}, there is a retraction $\fun{r}{Y}{S^1}.$ Consider the function
$\fun{R}{F_n(Y)}{F_n\left(S^1\right)}$ defined, for $A \in F_n(Y),$ by $R(A) = r(A).$
It is not difficult to see that $R$ is a well defined retraction. Since $X$ is contractible
(\cite[Theorem 1.2(21)]{Chara2}), $F_n(X)$ is contractible. Thus $F_n\left(S^1\right)$ is a retract
of the contractible space $F_n(Y),$ so $F_n\left(S^1\right)$ is contractible as well
(\cite[Theorem 13.2]{borsuk2}). However, in \cite{Wu} it is shown that there is no $n \in \m{N}$ so
that $F_n(S^1)$ is contractible. This contradiction shows that $Y$ does not contain a simple closed
curve. We conclude that $Y$ is a dendrite.
\end{proof}

Let $X$ be a dendrite and $K$ be a subcontinuum of $X.$ Define $\fun{r}{X}{K}$ as follows: $r(x) = x$ if
$x \in K$ and, otherwise, $r(x)$ is the unique point in $K$ such that $r(x)$ is a point of every arc
in $X$ from $x$ to any point of $K$ (see \cite[Lemma 10.24]{Nadler}). In \cite[Lemma 10.25]{Nadler} it is
shown that $r$ is a retraction. Such function is called the \emph{first point map for} $K.$ We use this
function in the proof of the following result.

\begin{theorem} \label{main}
Let $X \in \mathcal{D}$ and $n \in \m{N}.$ If $Y$ is a continuum such that
$F_n(X) \thickapprox F_n(Y),$ then $Y \in \mathcal{D}.$
\end{theorem}
\begin{proof}
Since $X \thickapprox F_1(X),$ the result is true for $n = 1,$ so we consider that
$n \geq 2.$ By Theorem \ref{primer-paso}, $Y$ is a dendrite. Let us assume that the metric $d$ for
$Y$ is convex. If $p,q \in \m{R}^2,$ we denote by $[p,q]$ the straight line segment in $\m{R}^2$ joining $p$ and $q.$
We consider that $(p,q) = [p,q] - \{p,q\}.$ 

Assume, to the contrary, that $Y \notin \mathcal{D}.$ Then, by \cite[Theorem 3.3]{Paty}, $Y$ contains either a copy of
$$
F_\omega = [(-1,0),(1,0)] \cup \left(\bigcup_{m=1}^\infty \left[(0,0),\left(\frac{1}{m},\frac{1}{m^2}\right)\right]\right)
$$
\noindent or of
$$
W = [(-1,0),(1,0)] \cup \left( \bigcup_{m=1}^\infty \left[\left(-\frac{1}{m},0\right),\left(-\frac{1}{m},\frac{1}{m}\right)\right]\right).
$$
To simplify notation let us assume that either $F_\omega \subset Y$ or $W \subset Y.$ Note that
$(0,0) \in \cl[Y]{E(Y)} - E(Y).$ Let $x_1 = (0,0).$ Since $O(Y)$ is dense in $Y,$ we can take $n - 1$
points $x_2,x_3\ldots,x_n$ in $O(Y) \cap ((0,0),(1,0)).$ Let 
$$
B = \{x_1,x_2,\ldots,x_n\}.
$$
Let $\fun{h}{F_n(X)}{F_n(Y)}$ be a homeomorphism. We will proceed as follows: after proving Claim 1, we 
consider the cases $h^{-1}(B) \cap E_a(X) = \emptyset$ and $h^{-1}(B) \cap E_a(X) \ne \emptyset.$ In both 
situations we will find a contradiction. Thus the assumption $Y \notin \mathcal{D}$ is not correct and, in
this way, the proof of the theorem will be complete.

\vskip .3cm

By Theorems \ref{e-0} and \ref{en-denso}, $h(\mathcal{E}_n(X)) = \mathcal{E}_n(Y)$ and $\mathcal{E}_n(Y)$ is 
dense in $F_n(Y).$

\vskip .2cm

Take $\delta > 0$ such that ${B_Y(x_i,\delta)} \cap {B_Y(x_j,\delta)} = \emptyset$ 
for each $i,j \in \{1,2,\ldots,n\}$ with $i \ne j.$ 

\vskip .3cm

\textbf{Claim 1}. For each open neighborhood $\mathcal{V}$ of $B$ in $F_n(Y)$ with $\mathcal{V} \subset
B_{F_n(Y)}(B,\delta),$ the set $\mathcal{V} \cap \mathcal{E}_n(Y)$ has infinitely many components.

\vskip .3cm

To show Claim 1, let $\mathcal{V}$ be an open neighborhood of $B$ in $F_n(Y)$ such that $\mathcal{V} \subset
B_{F_n(Y)}(B,\delta).$ Let $0 <\ve < \delta$ be such that the sets $B_{Y}\left(x_1,\ve\right),$
$B_{Y}\left(x_2,\ve\right), \ldots, B_{Y}\left(x_n,\ve\right)$ are pairwise disjoint and
$$
\langle B_{Y}\left(x_1,\ve\right),B_{Y}\left(x_2,\ve\right),\ldots,
B_{Y}\left(x_n,\ve\right)\rangle_n \subset \mathcal{V}.
$$
Since $x_1 \in B_{Y}\left(x_1,\ve\right)$ and either $F_\omega \subset Y$ or $W \subset Y,$ there
exists $N \in \m{N}$ such that either
\begin{equation} \label{chase1}
\bigcup_{m=N}^\infty \left[(0,0),\left(\frac{1}{m},\frac{1}{m^2}\right)\right]
\subset B_{Y}\left(x_1,\ve\right)
\end{equation}
\noindent or
\begin{equation} \label{chase2}
\bigcup_{m=N}^\infty \left[\left(-\frac{1}{m},0\right),\left(-\frac{1}{m},\frac{1}{m}\right)\right]
\subset B_{Y}\left(x_1,\ve\right).
\end{equation}

Also, since $Y$ is a dendrite, $x_1 \in \cl[Y]{E(Y)} - E(Y)$ and, according the case, the sequences
$\left(\left(\frac{1}{m},\frac{1}{m^2}\right)\right)_m$ or
$\left(\left(-\frac{1}{m},0\right)\right)_m$ and $\left(\left(-\frac{1}{m},\frac{1}{m}\right)\right)_m$
converge to $x_1,$ we can take $N$ so that, for every $m \geq N,$ if (\ref{chase1}) holds then the component
of $B_{Y}\left(x_1,\ve\right) - \{x_1\}$ that contains $\left(\frac{1}{m},\frac{1}{m^2}\right)$
coincides with the component of $Y - \{x_1\}$ that contains $\left(\frac{1}{m},\frac{1}{m^2}\right);$ and
if (\ref{chase2}) holds, then the component of $B_{Y}\left(x_1,\ve\right) -
\left\{\left(-\frac{1}{m},0\right)\right\}$ that contains $\left(-\frac{1}{m},\frac{1}{m}\right)$ coincides
with the component of  $Y - \left\{\left(-\frac{1}{m},0\right)\right\}$ that contains
$\left(-\frac{1}{m},\frac{1}{m}\right).$

Given $m \geq N$ we define $Z_m$ as follows: if (\ref{chase1}) holds, then $Z_m$ is the
component of $Y - \{x_1\}$ that contains $\left(\frac{1}{m},\frac{1}{m^2}\right)$ and,
if (\ref{chase2}) holds, then $Z_m$ is the component of $Y - \left\{\left(-\frac{1}{m},0\right)\right\}$
that contains $\left(-\frac{1}{m},\frac{1}{m}\right).$ Since each $Z_m$ is open in $Y$ and $\mathcal{E}_n(Y)$ is
dense in $F_n(Y),$ by Theorem \ref{arco-libre}, there is a free arc $A_m$ of $Y$ contained in $Z_m.$ Note that
$\{\inte[Y]{A_m} \colon m \geq N\}$ is a sequence of pairwise disjoint open connected subsets of
$B_{Y}\left(x_1,\ve\right).$

Given $i \in \{2,3,\ldots,n\},$ since $B_{Y}\left(x_i,\ve\right)$ is open in $Y$ and $\mathcal{E}_n(Y)$
is dense in $F_n(Y),$ by Theorem \ref{arco-libre}, there is a free arc $J_i$ of $Y$ contained in
$B_{Y}\left(x_i,\ve\right).$ Note that $\inte[Y]{J_2},\ldots,\inte[Y]{J_n}$ is a finite
sequence of pairwise disjoint open connected subsets of $Y.$

For $m \geq N$ define
$$
\mathcal{A}_m = \langle \inte[Y]{A_m}, \inte[Y]{J_2},\ldots, \inte[Y]{J_n}\rangle_n.
$$
Since $\inte[Y]{A_m}, \inte[Y]{J_2},\ldots, \inte[Y]{J_n}$ are open connected subsets of $Y,$
by Theorem \ref{n-celda}, $\mathcal{A}_m$ is an open connected subset of $F_n(Y).$ Since
$\inte[Y]{A_m} \cap \inte[Y]{A_k} = \emptyset$ if $m \ne k,$ we have $\mathcal{A}_m \cap \mathcal{A}_k =
\emptyset.$ Given $C \in \mathcal{A}_m,$ by Theorem \ref{n-celda}, the set $\langle A_m, J_2,\ldots,
J_n\rangle_n$ is an $n$-cell in $F_n(Y)$ that contains $C$ in its interior. Thus $C \in \mathcal{E}_n(Y),$ so
$\mathcal{A}_m \subset \mathcal{E}_n(Y).$ Moreover, we have
$$
\mathcal{A}_m \subset \langle B_{Y}\left(x_1,\ve\right),B_{Y}\left(x_2,\ve\right),\ldots,
B_{Y}\left(x_n,\ve\right)\rangle_n \subset \mathcal{V},
$$
\noindent so $\mathcal{A}_m \subset \mathcal{V} \cap \mathcal{E}_n(Y).$ 

Let $\mathcal{B}_m$ be the component of $\mathcal{V} \cap \mathcal{E}_n(Y)$ that contains
$\mathcal{A}_m.$ We claim that $\mathcal{B}_m \cap \mathcal{B}_k = \emptyset$ for different
$m,k \geq N.$ Assume, to the contrary, that $\mathcal{B}_m = \mathcal{B}_k.$ Let $D_1 \in \mathcal{A}_m$ 
and $D_2 \in \mathcal{A}_k.$ Since $F_n(Y)$ is locally connected and $\mathcal{B}_m$ is a component of the 
open subset $\mathcal{V} \cap \mathcal{E}_n(Y)$ of $F_n(Y),$ the set $\mathcal{B}_m$ is arcwise connected. 
Then there is an arc $\fun{\alpha}{[0,1]}{\mathcal{B}_m}$ such that $\alpha(0) = D_1$ and $\alpha(1) = D_2.$ Let
$$
K = \bigcup \{\alpha(t) \colon t \in [0,1]\}.
$$
Given $j \in \{1,2,\ldots,n\},$ let $K_j = K \cap B_Y(x_j,\delta).$ Since $\alpha([0,1])$ is connected in
$F_n(Y),$ the subset $K$ of $Y$ has at most $n$ components (\cite[Lemma 6.1]{Char-Ill}). Since
$D_1,D_2 \subset K,$ the sets $B_Y(x_1,\delta),B_Y(x_2,\delta),\ldots, B_Y(x_n,\delta)$
are pairwise disjoint and $D_i \cap B_Y(x_j,\delta) \ne \emptyset,$ for each $i \in \{1,2\}$ and every 
$j \in \{1,2,\ldots,n\},$ it follows that $K_1,K_2,\ldots,K_n$
are the components of $K.$ Note that $K_1 \cap A_m \ne \emptyset$ and $K_1 \cap A_k \ne \emptyset$ so, if
(\ref{chase1}) holds, then $x_1 \in K_1$ and, if (\ref{chase2}) holds, then $\left(-\frac{1}{m},0\right) \in K_1.$
This implies that $K \cap R(Y) \ne \emptyset,$ so one element of $\mathcal{B}_m$ contains a ramification point
of $Y.$ This contradicts Theorem \ref{core}. Hence $\mathcal{B}_m \cap \mathcal{B}_k = \emptyset.$ 

Therefore $\mathcal{V} \cap \mathcal{E}_n(Y)$ has infinitely many components. This completes the
proof of Claim 1.

\vskip .3cm

Let us assume that $h^{-1}(B) \cap E_a(X) = \emptyset.$ Then, by Theorem \ref{importante1},
there exists a basis $\mathfrak{B}_X$ of open neighborhoods of $h^{-1}(B)$ in $F_n(X)$ such
that, for each $\mathcal{U} \in \mathfrak{B}_X$, the set $\mathcal{U} \cap \mathcal{E}_n(X)$
is nonempty and has a finite number of components. Let $\mathfrak{B}_Y = \{h(\mathcal{U}) \colon
\mathcal{U} \in \mathfrak{B}_X\}.$ Then $\mathfrak{B}_Y$ is a basis of open neighborhoods of $B$
in $F_n(Y)$ such that, for each $\mathcal{V} \in \mathfrak{B}_Y,$ the set $\mathcal{V} \cap \mathcal{E}_n(Y)$
is nonempty and has a finite number of components. Let $\mathcal{V} \in \mathfrak{B}_Y$ such that
$\mathcal{V} \subset B_{F_n(Y)}(B,\delta).$ By Claim 1) the set $\mathcal{V} \cap \mathcal{E}_n(Y)$ has infinitely
many components. This is a contradiction. 

\vskip .3cm

Let us assume now that $h^{-1}(B) \cap E_a(X) \ne \emptyset.$ Then,
by Theorem \ref{vecindades contraibles}, there is a basis $\mathfrak{B}$ of open neighborhoods of
$h^{-1}(B)$ in $F_n(X)$ such that, for each $\mathcal{U} \in \mathfrak{B},$ the set
$\mathcal{U} - \{h^{-1}(B)\}$ is contractible. Let
$\mathfrak{C} = \{h(\mathcal{U}) \colon \mathcal{U} \in \mathfrak{B}\}.$ Then
$\mathfrak{C}$ is a basis of open neighborhoods of $B$ in $F_n(Y)$ such that, for each
$\mathcal{V} \in \mathfrak{C},$ the set $\mathcal{V} - \{B\}$ is contractible.

Let $A = [(-1,0),(1,0)].$ Note that $A$ is an arc in $Y$ such that $x_1=(0,0) \in ((-1,0),(1,0)).$

\vskip .3cm
\textbf{Claim 2}. There is a retraction $\fun{r}{Y}{A}$ such that $r^{-1}(x_i) = \{x_i\},$ for each
$i \in \{1,2,\ldots,n\}.$
\vskip .3cm

To show Claim 2, let $\fun{r_1}{Y}{A}$ be the first point map for $A.$ By \cite[Lemma 10.25]{Nadler},
$r_1$ is a retraction. Given $i \in \{2,3,\ldots,n\},$ since $x_i \in O(Y),$ we have
$r_1^{-1}(x_i) = \{x_i\}.$ If $x_1 \in O(Y),$ then $r_1^{-1}(x_1) = \{x_1\}$ and $r_1$ has
the required properties. If $x_1 \notin O(Y),$ then $r_1^{-1}(x_1) = \left\{y \in Y \colon
[y,x_1] \cap A = \{x_1\}\right\}.$ Let $A_0 = [(-1,0),(0,0)].$ Given $y \in r_1^{-1}(x_1),$ if
$d(x_1,y) \leq d(x_1,(-1,0)),$ there is a unique $z_y \in A_0$ such that $d(x_1,z_y) = d(x_1,y).$ Then
we can define a function $\fun{r_2}{r_1^{-1}(x_1)}{A_0}$ so that:
$$
r_2(y) =
\left\{
  \begin{array}{rr}
    z_y, & \hbox{ if } d(x_1,y) \leq d(x_1,(-1,0)); \\
    (-1,0), & \hbox{ if } d(x_1,y) \geq d(x_1,(-1,0)).
  \end{array}
\right.
$$
It is not difficult to see that $r_2$ is a well defined continuous function such
that $r_2^{-1}(x_1) = x_1.$ Now define $\fun{r}{Y}{A}$ so that, if $y \in Y,$ then:
$$
r(y) =
\left\{
  \begin{array}{ll}
    r_1(y), & \hbox{ if } y \notin r_1^{-1}(x_1); \\
    r_2(y), & \hbox{ if } y \in r_1^{-1}(x_1).
  \end{array}
\right.
$$

Then $r$ is a retraction such that $r^{-1}(y)=\{x_{i}\},$ for each $i\in
\{1,2,...,n\}.$ This proves Claim 2.

\vskip .3cm

Let $\fun{r}{Y}{A}$ as in Claim 2. Define $\fun{R}{F_n(Y)}{F_n(A)},$ at $D \in F_n(Y),$ by
$R(D) = r(D).$ Then $R$ is a retraction such that $R^{-1}(B) = \{B\}.$ For each $\ve > 0$ with $\ve<\delta$, 
let $U_i^\ve = B_Y(x_i,\ve)$ for $i \in \{1,2,\ldots,n\}$ and 
$\mathcal{U}^{\ve} = \left\langle U_1^{\ve},U_2^{\ve},\ldots,U_n^{\ve}\right\rangle_n$.

\vskip .3cm
\textbf{Claim 3}. For each $\ve>0$ with $\ve<\delta$, the set $R(\mathcal{U}^\ve)$ is a connected
open subset of $F_n(A)$ homeomorphic to the Euclidean space $\m{R}^n$.
\vskip .3cm

Considering the sets $U_1^{\ve},U_2^{\ve},\ldots,U_n^{\ve}$ are pairwise disjoint, it is not
difficult to prove that
$R(\mathcal{U}^{\ve}) = \left\langle r(U_1^{\ve}),r(U_2^{\ve}),\ldots,r(U_n^{\ve})\right\rangle_n.$
Since the metric for $Y$ is convex, by the definition of $r,$ $r(U_1^{\ve}),r(U_2^{\ve}),\ldots,
r(U_n^{\ve})$ are open connected subsets of $A.$ Thus $R(\mathcal{U}^{\ve})$ is an open connected
subset of $F_n(A).$ Moreover, since $A$ is an arc, by Theorem \ref{n-celda},
$R(\mathcal{U}^{\ve})$ is homeomorphic to $\m{R}^n.$ This proves Claim 3.

\vskip .3cm

Now we are ready to show the final argument. Fix $\gamma>0$ with $\gamma<\delta$ and
take $\mathcal{V} \in \mathfrak{C}$ such that $B \in \mathcal{V} \subset \mathcal{U}^{\gamma}.$
Now let $\sigma>0$ such that $B \in \mathcal{U}^{\sigma} \subset \mathcal{V} \subset \mathcal{U}^{\gamma}.$ Since
$R$ is a retraction, $\mathcal{V} - \{B\}$ is contractible and $R(\mathcal{V} - \{B\}) = R(\mathcal{V}) - \{B\},$
then the set $R(\mathcal{V}) - \{B\}$ is contractible.

Since $B \in \mathcal{U}^{\sigma}\subset \mathcal{V} \subset \mathcal{U}^{\gamma},$ by the definition of $R,$
$B \in R(\mathcal{U}^{\sigma}) \subset R(\mathcal{V}) \subset R(\mathcal{U}^{\gamma}).$
Thus, by Claim 3, $R(\mathcal{U}^{\sigma})$ is an open neighborhood of $B$ homeomorphic to $\m{R}^n$
and contained in the set $R(\mathcal{V})$. Then there exists an $n$-cell $G$ such that
$B \in G \subset R(\mathcal{V})$ and $B \notin \partial G,$ where $\partial G$ is the manifold boundary of $G$. By Claim 3, 
$R(\mathcal{U}^{\gamma})$ is also homeomorphic to $\m{R}^n$ so there is a retraction
$\fun{S}{R(\mathcal{U}^{\gamma}) - \{B\}}{\partial G}.$ Then
$\fun{S_{|R(\mathcal{V}) - \{B\}}}{R(\mathcal{V}) - \{B\}}{\partial G}$ is also a
retraction. Since $R(\mathcal{V}) - \{B\}$ is contractible, the set $\partial G$ is
contractible. This is a contradiction to \cite[p. 37]{Hurewicz} that came from
the assumption that $h^{-1}(B) \cap E_a(X) \ne \emptyset.$

Since both cases $h^{-1}(B) \cap E_a(X) = \emptyset$ and $h^{-1}(B) \cap E_a(X) \ne \emptyset$
produced a contradiction, the assumtion that $Y \notin \mathcal{D}$ is not correct. Therefore 
$Y \in \mathcal{D}.$
\end{proof}
\textbf{Acknowledgement. }We would like to thank professor Alejandro Illanes
for his suggestions for this paper. We would also like to thank the referee
for his/her fruitful comments for the improvement of this paper.

\end{document}